\documentclass[reqno,a4paper, 11pt]{amsart}

\usepackage[utf8]{inputenc}

\usepackage[a4paper=true,pdfpagelabels]{hyperref}
\usepackage{graphicx}

\usepackage{amsfonts,epsfig}
\usepackage{latexsym}
\usepackage{mathabx}
\usepackage{amsmath}
\usepackage{amssymb}

\usepackage{color}

\newtheorem{theorem}{Theorem}
\newtheorem{lemma}[theorem]{Lemma}

\newtheorem{lettertheorem}{Theorem}
\newtheorem{letterlemma}[lettertheorem]{Lemma}

\theoremstyle{definition}

\theoremstyle{remark}

\numberwithin{equation}{section}

\newcommand{\hg}{\H_{\g}}
\newcommand{\cw}{C_{\om}}

\newcommand{\D}{\mathbb{D}}
\newcommand{\DD}{\widehat{\mathcal{D}}}
\newcommand{\Dd}{\widecheck{\mathcal{D}}}
\newcommand{\M}{\mathcal{M}}
\newcommand{\DDD}{\mathcal{D}}

\newcommand{\N}{\mathbb{N}}

\renewcommand{\phi}{\varphi}

\def\a{\alpha}       \def\b{\beta}        \def\g{\gamma}
           
     \def\om{\omega}      
              
                  \def\z{\zeta}

\def\omg{\widehat{\omega}}

\def\fg{\widehat{f}}
\def\C{\mathcal{ C}}

\renewcommand{\H}{\mathcal{H}}

\newenvironment{Prf}{\noindent{\emph{Proof of}}}
{\hfill$\Box$ }

\makeatletter
\newcommand{\opnorm}{\@ifstar\@opnorms\@opnorm}
\newcommand{\@opnorms}[1]{%
  \left|\mkern-1.5mu\left|\mkern-1.5mu\left|
   #1
  \right|\mkern-1.5mu\right|\mkern-1.5mu\right|
}
\newcommand{\@opnorm}[2][]{%
  \mathopen{#1|\mkern-1.5mu#1|\mkern-1.5mu#1|}
  #2
  \mathclose{#1|\mkern-1.5mu#1|\mkern-1.5mu#1|}
}
\makeatother

\begin{document}

\title[Generalized Cesàro operator]{Generalized Cesàro operator acting on Hilbert spaces of analytic functions.}


\author{Alejandro Mas}
\address{Departamento de Análisis Matemático, Universidad de Valencia, 46100 Burjassot, Spain} \email{alejandro.mas@uv.es}

\author{Noel Merchán}
\address{Departamento de Matemática Aplicada, Universidad de M\'alaga, Campus de
Teatinos, 29071 M\'alaga, Spain} \email{noel@uma.es}

\author{Elena de la Rosa}
\address{Departamento de An\'alisis Matem\'atico, Universidad de M\'alaga, Campus de
Teatinos, 29071 M\'alaga, Spain} 
\email{elena.rosa@uma.es}

\thanks{This research was supported in part by Ministerio de Ciencia e Innovaci\'on, Spain, project PID2022-136619NB-I00; La Junta de Andaluc{\'i}a, project FQM210; The first author was partially supported by Grant PID2019-106870GB-I00 from Ministerio de Ciencia e Innovación.}

\subjclass[26A33, 30H30]{26A33,30H30}

\maketitle

\begin{abstract}
Let $\mathbb{D}$ denote the unit disc in $\mathbb{C}$. We define 
 the generalized Cesàro operator as follows
 $$
  C_{\omega}(f)(z)=\int_0^1 f(tz)\left(\frac{1}{z}\int_0^z B^{\omega}_t(u)\,du\right)\,\omega(t)dt,$$
  where  $\{B^{\omega}_\zeta\}_{\zeta\in\mathbb{D}}$ are the  reproducing kernels of the Bergman space $A^2_\omega$ induced by a radial weight $\omega$ in the unit disc $\mathbb{D}$. We study the action of the operator $C_{\om}$ on weighted Hardy spaces of analytic functions $\H_{\g}$, $\g>0$ and on general weighted Bergman spaces $A^2_{\mu}$.

  \end{abstract}
\section{Introduction}

Let $\H(\D)$ denote the space of analytic functions in the unit disc $\D=\{z\in\mathbb{C}:|z|<1\}$.

For $\g>0$, let $\H_{\g}$ denote the Hilbert space of analytic functions in $\D$ such that its reproducing kernels are given by 
$$K_{\om}(z)=\frac{1}{(1-z\bar{\om})^\g}=\sum\limits_{n=0}^{\infty} \g(n)(\bar{\om} z)^{n}, \, z, \, \om \in \D.$$
 It is clear that the sequence $\g(n)$ is given by
$\g(0)=1$, $\g (1)=\g$ and $\g(n)=\frac{\Gamma(n+\g)}{\Gamma(\g) n!}$, $n \in \N$. Actually, this family of spaces are well-known: for $\g=1$ the space $\H_{\g}$ is the Hardy space $\H_1=H^2$ and $\g(n)=1$ for all $n \in \N$. For $\g >1$, $\H_{\g}$ consists of the standard weighted Bergman space $A^2_{\g-2}$ and for $\g <1$, it is the weighted Dirichlet space $\H_{\g}=D^2_{\g}$. 

Observe that for $\g=0$, the corresponding space would be the classical Dirichlet space $D^2$, so it is not included in the definition of the spaces $\H_{\g}$.

In other words, the Hilbert space $\H_{\g}$ consists of all the analytic functions such that

$$ \|f\|_{\H_{\g}}^2= |f(0)|^2+ \int_{\D} |f'(z)|^2 (1-|z|)^{\g} dA(z) <\infty,$$
where $dA(z)=\frac{dx\,dy}{\pi}$ is the normalized area measure on $\D$.
Moreover, a simple observation yields an equivalent norm in terms of the coefficients of an analytic function $f$. If $f(z)=\sum \limits_{k=0}^{\infty} \fg(k) z^k$,

$$ \|f\|_{\H_{\g}}^2 \asymp \sum\limits_{n=0}^{\infty} | \fg(n)|^2 (n+1)^{1-\g}.$$

Further, we can consider more general weighted Bergman spaces than the ones defined by $\H_{\g}$ with $\g>1$. For a  nonnegative function  $\om  \in L^1_{[0,1)}$, the extension to $\D$, defined by $\om(z)=\om(|z|)$ for all $z\in\D$, is called a radial weight.
 Let $A^2_{\om}$ denote the weighted Bergman space of $f\in\H(\D)$ such that $\|f\|_{A^2_\omega}^2=\int_\D|f(z)|^2\omega(z)\,dA(z)<\infty$.
 Throughout this paper we assume $\widehat{\om}(z)=\int_{|z|}^1\om(s)\,ds>0$ for all $z\in\D$, for otherwise $A^2_\om=\H(\D)$.

 For any radial weight, the convergence in $A^2_{\om}$ implies the uniform convergence in compact subsets, so the point evaluations $L_z$ are bounded linear functionals in $A^2_{\om}$ and by the Riesz Representation Theorem there exist Bergman reproducing kernels $B^\om_z\in A^2_\om$ such that
$$ L_z(f)=f(z)=\langle f, B_z^{\om}\rangle_{A^2_\om}=\int_{\D}f(\z)\overline{B_z^{\om}(\z)}\om(\z)dA(\z),\quad f \in A^2_{\om}.$$

For a complex sequence $\{a_k\}_{k=0}^{\infty}$, the classic Cesàro operator is defined as follows: 
$$ \C (\{a_k\})=\left\{\frac{1}{n+1}\sum\limits_{k=0}^{n} a_k \right\}_{n=0}^{\infty}.$$
It is well known that the Cesàro operator is bounded on $l^p$, $1<p<\infty$. This result was mostly showed by Hardy, whose main aim was to provide a simpler proof of the Hilbert inequality in \cite{Hardy1920, Hardy1925} and Landau \cite{Landau}, whose  contribution was obtaining the sharp constant in the inequality, that is, the norm of the operator, among other authors. 

Further, it can be considered as an operator between analytic functions by identifying each analytic function with its Taylor coefficients as follows: for $f \in \H(\D)$, $f(z)=\sum \limits_{k=0}^{\infty} \fg(k) z^k$, 
$$ \C (f)(z)=\sum\limits_{n=0}^{\infty} \left(\frac{1}{n+1}\sum\limits_{k=0}^{n} \fg(k) \right) z^n, \quad z \in \D. $$

Observe that it defines an analytic function, and a simple calculation gives the following integral representation:

\begin{equation}
\label{intclassicCesaro}
\C(f)(z)=\int_0^1 f(tz)\frac{1}{1-tz} dt, \quad z \in \D.
\end{equation}
 
This operator is bounded on $H^p$, $0<p<\infty$. This result has been showed by several authors and on different ways such as Hardy \cite{Hardy29}, Siskakis \cite{SiskakisHp, SiskakisH1}, Miao \cite{Miao}, Stempak \cite{Stempak} and Andersen \cite{Andersen}, among others.

The boundedness of the Cesàro operator on Bergman spaces was studied in \cite{Andersen} and \cite{SisBergman1} where  it is shown that the Cesàro operator is bounded from $A^p_{\a}$ into itself if $p>0$ and $\a >-1$.

Regarding Dirichlet spaces, Galanopoulos \cite{Ga2001} proved that it is bounded on the weighted Dirichlet spaces $D^2_{\a}$ if $0<\a<1$.

Due to the historical magnitude of this classical operator and the authors that have been working on it, different generalizations have been raised during the last decades (\cite{Blascocomplex, GaGiMer22, GaGirMasMer, Stempak}). Bearing in mind the formula \eqref{intclassicCesaro}, we are interested in replacing the kernel $\frac{1}{1-tz}$ of the integral representation with a more general kernel. In that sense, we are going to focus on the following generalization of the kernel induced by radial weights, which was previously introduced in works regarding the Hilbert operator \cite{ MerPeldelaR, PeldelaRosa21}.

For a radial weight $\omega$, we consider the generalized Cesàro operator
\begin{equation}\label{eq:i1}
    C_{\omega}(f)(z)=\int_0^1 f(tz)\left(\frac{1}{z}\int_0^z B^{\om}_t(\z)d\z\right)\,\om(t)dt,
\end{equation}
where $\{B^\om_z\}_{z\in\D}\subset A^2_\om$ are the Bergman reproducing kernels of $A^2_\omega$.
Notice that this operator is well defined for any analytic function and the choice $\om=1$ gives \eqref{intclassicCesaro}. 

One of the first and main obstacles that we find when dealing with the operator \eqref{eq:i1} is that Bergman reproducing kernels have not an explicit formula in general (this is not the case for standard weights $\nu_{\a}(z)=(1-|z|)^{\a}$, $\a>-1$, since Bergman reproducing kernels induced by $\nu_{\a}$ have nice properties and they can be written as $B^{\nu_{\a}}_z(\z)=(1-\overline{z}\z)^{-(2+\a)}$ ). Consequently, we are forced to use that for any radial weight $\om$ they can be written as 
 $B^\om_z(\z)=\sum \overline{e_n(z)}e_n(\z)$  for each orthonormal basis $\{e_n\}$ of $A^2_\om$, and therefore, by using the normalized monomials as basis, we can obtain the following representation in terms of the odd moments of the weight, denoted by $\om_{2n+1}$:
\begin{equation}\label{eq:B}
B^\om_z(\z)=\sum_{n=0}^\infty\frac{\left(\overline{z}\z\right)^n}{2\om_{2n+1}}, \quad z,\z\in \D.
\end{equation}
 In general, from now on, we will write $\om_x=\int_0^1r^x\om(r)\,dr$ for all $x\ge0$.  In addition, we can write the norm of $A^2_\omega$ in terms of the Taylor coefficients of an analytic function as follows, 
$$ \| f \|^2_{A^2_\omega} = \sum_{n=0}^{\infty} 2 \omega_{2n+1} |\fg(n)|^{2}. $$

 The primary purpose of this paper is to describe the radial weights $\omega$ so that $C_\omega$ is bounded on $\H_{\g}$, for $\g>0$ and on general weighted Bergman spaces. It is worth mentioning that just as Galanopoulos \cite{Ga2001} pointed out that $\C$ is not bounded in the Dirichlet space $D ^2$, this fact is true not only for $\om=1$ but also for any radial weight. Indeed, by using the formula \eqref{eq:B}, for any radial weight $\om $,
 $$ C_{\om}(1)(z)=\sum \limits_{n=0}^{\infty}\frac{\om_n}{2 (n+1)\om_{2n+1}} z^n,$$
 so, since the moments of a radial weight form a decreasing sequence, we have  
 $\|C_{\om}(1)\|_{D^2}^2 \asymp \sum \limits_{n=0}^{\infty} \frac{\om_n ^2}{4 (n+1)\om_{2n+1}^2 }\geq  \sum \limits_{n=0}^{\infty} \frac{1}{4 (n+1)}$ which implies that $C_{\om}(1)$ does not belong to $D ^2$.

Before stating the main result of the paper, we need to introduce some notation and definitions.  A radial weight $\omega$
 belongs to the class~$\DD$ if there exists $C=C(\omega)>1$ such that
 $\widehat{\om}(r)\le C\widehat{\om}(\frac{1+r}{2})$ for all $0\le r <1$. 
This condition implies a restriction on the decay of the weight, for example, if $\om \in \DD$, $\om$ cannot decrease exponentially. However, every increasing weight belongs to $\DD$, and weights of $\DD$ admit an oscillatory behaviour.
The study of the intrinsic nature of this class of weights entails a considerable difficulty, which has led to a deep research for years, collected in works such as \cite{PelSum14,PR19,memoirs}.

A radial weight $\om\in\Dd$ if there exist $K=K(\om)>1$ and $C=C(\om)>1$ such that $\widehat{\om}(r)\ge C\widehat{\om}\left(1-\frac{1-r}{K}\right)$ for all $0\le r<1$. We write the class $\DDD=\DD\cap\Dd$.
Observe that standard weights $v_{\a}=(1-|z|)^{\a}$, $\a>-1$ belong to the class $\DDD$, which means that $H_{\g}=A^2_{\g-2}$, $\g>1$ are particular cases of weighted Bergman spaces $A^2_{\mu}$, $\mu \in \DDD$.

Moreover, a radial weight $\om\in\M$ if there exist constants $C=C(\om)>1$ and $K=K(\om)>1$ such that $\om_{x}\ge C\om_{Kx}$ for all $x\ge1$. 
Peláez and Rättyä showed that the classes $\Dd$ and $\M$ are closely related. They recently proved that $\Dd\subset \M$ \cite[Proof of Theorem~3]{PR19} but $\Dd \subsetneq \M$ \cite[Proposition~14]{PR19}. However, \cite[Theorem~3]{PR19} shows that $\DDD=\DD\cap\Dd=\DD\cap\M$. The theory of these classes of weights has been basically developed by these authors in the work \cite{PR19}, and they have shown that these classes of weights arise on a natural way in significant questions of the operator theory and the weighted Bergman spaces. For instance, $\DDD$ describes the radial weights such that the following Littlewood-Paley formula holds
$$\|f\|_{A^p_\om}^p\asymp\int_\D|f^{(n)}(z)|^p(1-|z|)^{np}\om(z)\,dA(z)+\sum_{j=0}^{n-1}|f^{(j)}(0)|^p,\; f\in\H(\D), 0<p<\infty, \; n \in \N$$
\noindent or the radial weights such that
 $ P_\om(f)(z)=\int_{\D}f(\z) \overline{B^\om_{z}(\z)}\,\om(\z)dA(\z)$ is bounded and onto from $L^{\infty}$ to the Bloch space, among other important results.

\begin{theorem}
\label{weighted H spaces}
Let $\om$ be a radial weight, $\g>0$. Then $C_{\om}: \H_{\g} \to \H_{\g} $ is bounded if and only if $\om \in \DDD$.
\end{theorem}

 The underlying nature of the spaces $\H_{\g}$ that we are considering and as far as we know, the almost unique formula for the Bergman reproducing kernels \eqref{eq:B} lead us to address the problem by working on coefficients, so an appropriate expression for $C_{\om}$ in terms of coefficients plays a key role in this work. Let $f \in \H(\D)$, $f(z)=\sum \limits_{k=0}^{\infty} \fg(k) z^k$, by \eqref{eq:B} and a change of variable,
\begin{equation}
\label{coef C_om}
\begin{split}
C_{\om} (f)(z)&=\sum\limits_{n=0}^{\infty} \frac{1}{2(n+1)\om_{2n+1}}\left(\sum\limits_{k=0}^{\infty} \fg (k)\om_{n+k}z^{n+k}\right)\\
&= \sum\limits_{n=0}^{\infty}\left(\sum\limits_{k=0}^{n} \frac{\fg(k)}{2(n-k+1)\om_{2(n-k)+1}}\right)\om_n z^n.
\end{split}
\end{equation}
The proof of the Theorem \ref{weighted H spaces} for $\g=1$ draws strongly on accurate estimates of the moments $\om_{2(n-k)+1}$ and $\om_n$ and on the Carleson measures theory.

For $0<\g<1$, the Carleson measures description was solved in \cite{Stegenga1980}, but the innocent looking condition that characterize such measures is not easy to work with, so we are forced to appeal to Littlewood-Paley formulas for non radial weights, specifically whose $\nu$ on $\mathbb{D}$ which belongs to one of the Békollé classes $B_{p}(\a)$ for some $p>1$ and $\a >-1$. 

The proof of the case $\g>1$ is slightly simpler since it is not necessary to use the Carleson measures tool. Going further, we are able to characterize the boundedness of the Cesàro-type operator $C_{\om}$ in more general weighted Bergman spaces $A^2_{\mu}$, $\mu \in \DDD$.

\begin{theorem}
\label{weighted Bergman spaces}
Let $\mu$ and $\om$ be radial weights, $\mu \in \DDD$. Then $C_{\om}: A^{2}_{\mu} \to A^{2}_{\mu} $ is bounded if and only if $\om \in \DDD$.
\end{theorem}

Finally, we are able to show in Theorem \ref{compact} that there does not exist radial weight $\om$ such that $\cw:\hg\to\hg$, $\g >0$, is compact neither radial weight such that $\cw:A^{2}_{\mu}\to A^{2}_{\mu}$, $\mu \in \DDD$, is compact.

The letter $C=C(\cdot)$ will denote an absolute constant whose value depends on the parameters indicated
in the parenthesis, and may change from one occurrence to another.
We will use the notation $a\lesssim b$ if there exists a constant
$C=C(\cdot)>0$ such that $a\le Cb$, and $a\gtrsim b$ is understood
in an analogous manner. In particular, if $a\lesssim b$ and
$a\gtrsim b$, then we write $a\asymp b$ and say that $a$ and $b$ are comparable.

\section{Previous results.}

Before tackling with the proof of Theorems \ref{weighted H spaces} and \ref{weighted Bergman spaces}, we gather the following two lemmas with some descriptions of the classes of weights $\DD$ and $\M$, which are useful for our purposes. The next one can be found in \cite[Lemma~2.1]{PelSum14}.

\begin{letterlemma}
\label{caract. pesos doblantes}
Let $\om$ be a radial weight on $\D$. Then, the following statements are equivalent:
\begin{itemize}
    \item[(i)] $\om \in \DD$;
    \item[(ii)] There exist $C=C(\om)\geq 1$ and $\a_0=\a_0(\om)>0$ such that
    $$ \omg(r)\leq C \left(\frac{1-r}{1-t}\right)^{\a}\omg(t), \quad 0\leq r\leq t<1;$$
   for all $\a\geq \a_0$.
   \item[(iii)] $$ \int_0^1 s^x \om (s) ds\asymp \omg\left(1-\frac{1}{x}\right),\quad x \in [1,\infty);$$
   \item[(iv)] There exist $C=C(\om)>0$ and $\a=\a(\om)>0$ such that 
   $$ \om_x\leq C \left(\frac{y}{x}\right)^{\a}\om_y,\quad 0<x\leq y<\infty ;$$
    \item[(v)] $\sup\limits_{n\in \N}\frac{\om_n}{\om_{2n}}<\infty .$
\end{itemize}
\end{letterlemma}

The following result can be found in \cite[(2.16) and (2.17)]{PR19}. Denote $\om_{[\b]}(z)=(1-|z|)^{\b} \om (z)$.
\begin{letterlemma}
\label{caract M}
Let $\om$ be a radial weight. The following statements are equivalent:
\begin{itemize}
\item[(i)] $\om	\in \M$;
\item[(ii)] There exist $C=C(\om)>0$ and $\b_0=\b_0(\om)>0$ such that
$$\om_x\geq C \left(\frac{y}{x}\right)^{\b} \om_y, \quad 1 \leq x \leq y<\infty$$
for all $0<\b\le \b_0$;
\item[(iii)] For some (equivalently for each) $\b > 0$, there exists $C = C(\om, \b) > 0$ such that
$$\om_x\leq C x^{\b} (\om_{[\b]})_x,  \quad 1 \leq x <\infty.$$ 
\end{itemize}
\end{letterlemma}

Now, we are interested in the weights $\nu$ that satisfy the following equivalence, called Littlewood-Paley formula:
\begin{equation}\label{LP}
 \int_{\D}|f(z)|^{p} \nu (z)dA(z) \asymp |f(0)|^{p}+ \int_{\D}|f'(z)|^{p} (1-|z|^{2})^{p}\nu (z)dA(z). 
\end{equation}
These kind of estimations are useful not only to obtain equivalent norms in terms of derivatives but also due to their relation with bounded Bergman projections, and this is one of the reasons why it is a prominent topic in the operator theory on spaces of analytic functions \cite{AC, APR, BWZ, PP, PR19}. We are interested in the one proved in \cite{AC} en route to a description of the spectra of integration operators on weighted Bergman spaces, where Aleman and Constantin showed that \eqref{LP} holds for every weight $\nu$ on $\mathbb{D}$ which belongs to one of the Békollé classes $B_{p}(\a)$ for some $p>1$ and $\a >-1$.

A weight $\nu$ on $\mathbb{D}$ belong to the Bekollé class $B_{p}(\a)$, $p>1$ and $\a >-1$ if 
$$\left(\int_{S(\theta, h)} \nu dA_{\a}\right)\left(\int_{S(\theta, h)} \nu^{-\frac{p'}{p}} dA_{\a}\right)^\frac{p}{p'} \lesssim (A_{\a}(S(\theta, h)))^p$$ 
for any Carleson square
$S(\theta, h)=\{z=re^{i\a}: 1-h<r<1, \, |\theta -\a|<h/2\}$, $ \theta \in [0, 2\pi]$, $h \in (0,1),$
where $A_{\a}$ denote the measure given by $dA_{\a}=(\a+1)(1-|z|^2)^{\a}dA$ and $1/p+1/p'=1$.

Moreover a weight $\nu$ on $\mathbb{D}$ belongs to the class $B^{\star}_{1}(\eta)$, $\eta >-1$ if 
\[
\int_{\mathbb{D}} \frac{\nu(z)}{|1-\overline{a}z|^{\eta+2}} \, (1-|z|^{2})^{\eta}\, dA(z) \lesssim  \nu(a)
\]
for almost every $a\in \mathbb{D}$.

In fact, they proved not only the belonging to one of the Bekollé classes is sufficient condition in order that \eqref{LP} holds, but also it is necessary for sufficiently regular weights $\nu$ on $\mathbb{D}.$ 

\begin{lettertheorem} \cite[Theorem 3.2]{AC}
\label{L-P}
Let $\nu$ be a strictly positive weight $\nu\in C^1(\mathbb{D})$ which satisfies that $(1-|z|^{2})|\nabla \nu(z)| \leq k_\nu \nu(z)$ for some constant $k_\nu >0$ and all $z \in \D$. Then the following are equivalent:
\begin{itemize}
\item[(i)] The estimate \eqref{LP} holds for all $p>0$;
\item[(ii)] The estimate \eqref{LP} holds for some $p>0$;
\item[(iii)] $\frac{\nu}{(1-|z|)^{\a}}$ belongs to $B_{p}(\a)$ for some $p>1$ and $\a >-1$;
\item[(iv)] $\frac{\nu}{(1-|z|)^{\eta}}$ belongs to $B^{\star}_{1}(\eta)$ for some $\eta >-1$.
\end{itemize}
\end{lettertheorem}

\section{Proof of the main results.}

\begin{Prf}\emph{Theorem \ref{weighted H spaces}.}

Since the Bergman case will be deal with in a more general way in Theorem \ref{weighted Bergman spaces}, it is enough to prove the result for $0<\g \leq 1$.

Let us consider the following suitable formula for the generalized Cesàro operator \eqref{coef C_om},
$$ C_\omega(f)(z)=\sum_{n=0}^\infty\left(\sum_{k=0}^n \frac{\widehat{f}(k)}{2(n-k+1)\omega_{2(n-k)+1}}\right) \om_n z^n, \; z \in \D.$$

Firstly, assume $\om \in \DDD$ and note that it is enough proving that there exists a constant $C>0$ such that $\| C_{\om}(f)\|_{\H_{\g}}^2\leq C \|f\|_{\H_\g}^2$ for any function $f \in \H(\D)$ such that $\fg(n)\geq 0$, $n \in \N \cup \{0\}$. 

Now,  by Lemma \ref{caract. pesos doblantes}(iv),
\begin{equation*}
\begin{split}
\| C_{\om}(f)\|_{\H_{\g}}^2 &\asymp \sum_{n=0}^\infty \om_n^2 (n+1)^{1-\g} \left(\sum_{k=0}^n\frac{\widehat{f}(k)}{2(n-k+1)\omega_{2(n-k)+1}}\right)^2 
\\ & \lesssim  \sum_{n=0}^\infty \om_n^2 (n+1)^{1-\g} \left(\sum_{k=0}^n\frac{\widehat{f}(k)}{2(n-k+1)\omega_{n-k+1}}\right)^2
\end{split}
\end{equation*}
 and by Lemma \ref{caract M} (ii), there exists $0<\b<1$ such that $(n-k+1)^{\b}\om_{n-k+1}\gtrsim (n+1)^{\b}\om_{n+1}$ for all $k\leq n$, so
\begin{equation}
\label{i1sufH2}
 \| C_{\om}(f)\|_{\H_{\g}}^2 \lesssim \sum_{n=0}^\infty \frac{1}{(n+1)^{2\b+\g-1}} \left(\sum_{k=0}^n\frac{\widehat{f}(k)}{(n-k+1)^{1-\b}}\right)^2.
\end{equation}

Now, it is well known that $g(z)=\frac{1}{(1-z)^\b}=\sum_{n=0}^{\infty} \a_n z^n \in \H(\D)$, whose Taylor coefficients are given by $\a_n=\frac{\Gamma(n+\b)}{\Gamma(n+1)\Gamma(\b)}$, and folklore estimations for ratios of gamma functions yields $\a_n\asymp \frac{1}{(n+1)^{1-\b}}$. 
In addition, a simple observation yields
 $$\frac{f(z)}{(1-z)^{\b}}= \sum\limits_{n=0}^{\infty}\left(\sum\limits_{k=0}^n \fg (k) \a_{n-k} \right)z^n,\, z \in \D. $$ 

From now on, we will deal with the following two cases separately:

\textbf{Case $\g = 1$:}
Bearing in mind that $\|z^n\|_{A^{2}_{2\b-1}}^2\asymp \frac{1}{(n+1)^{2\b}}$, $n \in \N$ and \eqref{i1sufH2},
we deduce
$$ \| C_{\om}(f)\|_{H^2}^2 \lesssim \sum\limits_{n=0}^{\infty} \|z^n\|_{A_{2\b-1}^2}^2 \left(\sum\limits_{k=0}^n \fg (k) \a_{n-k}\right)^2 \lesssim \left\| \frac{f(z)}{(1-z)^{\b}}\right\|_{A_{2\b-1}^2}^2 \lesssim \|f \|_{H^2}^2,$$
where the last inequality holds if and only if $d \nu(z)=\frac{(1-|z|^2)^{2\b-1}}{|1-z|^{2\b}}dA(z)$ is a Carleson measure (see \cite[Theorem~9.3]{Duren}).

A direct computation using the Cauchy-Schwarz inequality shows that
\begin{equation*}
\begin{split}
\sup\limits_{a \in \D}\int_{\D} \frac{1-|a|^2}{|1-\bar{a}z|^2}d\nu(z) &\lesssim  \sup\limits_{a \in \D} ( 1-|a|^2) \int_0^1 (1-r^2)^{2\b-1}\left(\int_0^{2\pi}\frac{1}{|1-\bar{a}re^{i\theta}|^2|1-re^{i\theta}|^{2\b}}d\theta \right)dr 
\\& \lesssim \sup\limits_{a \in \D} \int_0^1 \frac{( 1-|a|^2) }{(1-|a|r)^{\frac{3}{2}}(1-r)^{\frac{1}{2}}}<\infty,
\end{split}
\end{equation*}
so by \cite[Lemma~6.1]{GiBMO}, $\nu$ is a Carleson measure and this finishes the proof of this case.

\textbf{Case $\g < 1$:} In this case, $\H_{\g}=D^2_\g$, and we need to use that $\|z^n\|_{D^{2}_{2\b+\g}}^2\asymp \frac{1}{(n+1)^{2\b+\g-1}}$ to deduce
\begin{equation}
\begin{split}
\| C_{\om}(f)\|_{\H_{\g}}^2  \lesssim \sum\limits_{n=0}^{\infty} \|z^n\|_{D_{2\b+\g}^2}^2 \left(\sum\limits_{k=0}^n \fg (k) \a_{n-k}\right)^2 
\lesssim \left\| \frac{f(z)}{(1-z)^{\b}}\right\|_{D_{2\b+\g}^2}^2 \lesssim I+II
\end{split}
\end{equation}
where 
$$ I= \int_{\D}\frac{|f'(z)|^2}{|1-z|^{2\b}} (1-|z|^2)^{\g+ 2\b} dA(z)$$
and 
$$II= \int_{\D}\frac{|f(z)|^2}{|1-z|^{2\b +2}} (1-|z|^2)^{\g+ 2\b} dA(z).$$
It is clear that $I \lesssim \|f\|^2_{D^2_{\g}}$. Therefore, the proof of the sufficiency for $0<\g<1$ boils down to prove the inequality
$$\int_{\D}\frac{|f(z)|^2}{|1-z|^{2\b +2}} (1-|z|^2)^{\g+ 2\b} dA(z)\lesssim \| f \|_{D^2_{\g}}^2.$$
In order to simplify notation, let us denote by $\nu(z)=\frac{ (1-|z|^2)^{\g+ 2\b}}{|1-z|^{2\b +2}}$. Next, we are going to verify that $\nu $ satisfies the hypothesis of Theorem \ref{L-P}. Firstly, it is not difficult to show by a computation that $\nu \in C^1(\D)$ and that satisfies the regularity condition $(1-|z|^{2})|\nabla \nu(z)| \leq k_\nu \nu(z)$.

Secondly, set $\eta = 2\g +2\b$ and we will prove that $\frac{\nu}{(1-|z|)^{2\g+2\b}}\in B^{\star}_{1}(2\g+2\b) $, i.e., we have to prove that
\[
\int_{\mathbb{D}} \frac{(1-|z|^{2})^{\g+2\b}}{|1-\overline{a}z|^{2\g+2\b+2}|1-z|^{2\b+2}} \, dA(z) \lesssim  \frac{\nu(a)}{(1-|a|^{2})^{2\g+2\b}}, 
\]
for almost every $a \in \D.$ Now let $b_n=1-\frac{1}{n}$. By Fatou's lemma and by using the estimates in \cite[Lemma 2.5]{OF2} we have that
\begin{align*}
\int_{\mathbb{D}} \frac{(1-|z|^{2})^{\g+2\b}}{|1-\overline{a}z|^{2\g+2\b+2}|1-z|^{2\b+2}} \, dA(z) &\leq \liminf_{n\to \infty} \int_{\mathbb{D}} \frac{(1-|z|^{2})^{\g+2\b}}{|1-\overline{a}z|^{2\g+2\b+2}|1-b_n z|^{2\b+2}} \, dA(z) \\
&\lesssim \liminf_{n\to \infty}  \frac{1}{(1-|a|^{2})^{\g}|1-b_n\overline{a}|^{2+2\b}}\\
&=\frac{1}{(1-|a|^{2})^{\g}|1-a|^{2+2\b}}\\
&= \frac{\nu(a)}{(1-|a|^{2})^{2\g+2\b}}.
\end{align*}
By applying Theorem \ref{L-P}, the proof of the sufficiency is finished.

Conversely, assume $C_{\om}: \H_{\g}\to \H_{\g}$ is bounded, and we are going to show first $\om \in \DD$. Now we consider the following family of functions $f_{N}(z)=\sum\limits_{n=0}^{N} (n+1)^{\frac{\g -1}{2}} z^n$, $N \in \N$. Then $\|f_N\|_{\H_{\g}}^2 \asymp (N+1)$ and 
\begin{equation*}
\begin{split}
\| C_{\om}(f_N)\|_{\H_{\g}}^2 & \asymp \sum\limits_{n=0}^{\infty} (n+1)^{1-\g} \om_n^2 \left(\sum_{k=0}^n \frac{\fg_N(k)}{2(n-k+1)\om_{2(n-k)+1}}\right)^2
\\& \geq \sum\limits_{n=7N}^{8N} (n+1)^{1-\g} \om_n^2 \left(\sum_{k=0}^N \frac{(k+1)^{\frac{\g -1}{2}}}{2(n-k+1)\om_{2(n-k)+1}}\right)^2
\\ & \gtrsim
\frac{\om_{8N}^2 }{\om_{12N}^2} \frac{1}{(N+1)^2}\sum_{n=7N}^{8N}(n+1)^{1-\g} \left(\sum_{k=0}^N  (k+1)^{\frac{\g -1}{2}} \right)^2,
\end{split}
\end{equation*}
for all $N \in \N$, hence, 
\begin{equation*}
\begin{split}
\| C_{\om}(f_N)\|_{\H_{\g}}^2 & \gtrsim \frac{\om_{8N}^2 }{\om_{12N}^2} \frac{(N+1)^{\g+1}}{(N+1)^2} \sum_{n=7N}^{8N}(n+1)^{1-\g} \geq \frac{\om_{8N}^2 }{\om_{12N}^2} (N+1), \; N\in \N.
\end{split}
\end{equation*}
Since $C_{\om}: \H_{\g}\to \H_{\g}$ is bounded, $\om_{8N}\lesssim \om_{12N}$, $n \in \N$ and this implies $\om \in \DD$ by Lemma \ref{caract. pesos doblantes}(v).

Now we will prove $\om \in \M$, which combined with $\om \in \DD$ implies $\om \in \DDD$ by  \cite[Theorem~3]{PR19}.

Consider the family of test functions 
$f_{N,M} (z)=\sum\limits_{n=0}^{MN}z^n, \,  N, \, M \in \N$.
We would like to point out that from now on, the letter $C=C(\g, \om)>0$ will denote a constant whose value depends on $\g>0$ and $\om$, but does not depend on $M$ or $N$, and may change from one occurrence to another.

 On the one hand, observe that 
$\|f_{N,M}\|_{\H_{\g}}^2 \leq C \sum\limits_{n=0}^{MN} (n+1)^{1-\g},$ and on the other hand, 
\begin{equation*}
\begin{split}
\| C_\omega(f_{M,N})\|^2_{\H_\g}&\geq C \sum\limits_{n=0}^{MN}(n+1)^{1-\g} \om_n^2 \left(\sum_{k=0}^n \frac{1}{(n-k+1)\omega_{2(n-k)+1}}\right)^2
\\ &
\geq C \om_{MN}^2  \sum\limits_{n=0}^{MN}(n+1)^{1-\g} \left(\sum_{k=0}^n \frac{1}{(k+1)\omega_{k}}\right)^2, \; M,N \in \N ,
\end{split}
\end{equation*}
hence,

$$ \om_{MN}^2  \left(\frac{1}{\sum\limits_{n=0}^{MN}(n+1)^{1-\g} }  \sum\limits_{n=0}^{MN}(n+1)^{1-\g} \left(\sum_{k=0}^n \frac{1}{(k+1)\omega_{k}}\right)^2 \right) \leq C , \; M,N \in \N .$$
Therefore, by using Jensen inequality,
$$ \om_{MN} \left(\frac{1}{\sum\limits_{n=0}^{MN}(n+1)^{1-\g} }  \sum\limits_{n=0}^{MN}(n+1)^{1-\g} \left(\sum_{k=0}^n \frac{1}{(k+1)\omega_{k}}\right) \right) \leq C,  \; M,N \in \N , $$
so,
$$ \sum_{k=N}^{MN}  \frac{1}{(k+1)} \sum\limits_{n=k}^{MN}(n+1)^{1-\g} \leq C \frac{\om_N}{\om_{MN}} \sum\limits_{n=0}^{MN}(n+1)^{1-\g},  \; M,N \in \N.
$$

Now we are going to show that there exists a sufficiently large $M \in \N $ and $C'>1$ such that $$\frac{1}{C}\frac{1}{\left(\sum\limits_{n=0}^{MN}(n+1)^{1-\g}\right)}\left(\sum\limits_{k=N}^{MN}  \frac{1}{(k+1)} \sum\limits_{n=k}^{MN}(n+1)^{1-\g}\right)> C' \text{ for all }N \in \N .$$ 

Indeed,
 $$\frac{1}{\sum\limits_{n=0}^{MN}(n+1)^{1-\g}}\left(\sum\limits_{k=N}^{MN}  \frac{1}{(k+1)} \sum\limits_{n=k}^{MN}(n+1)^{1-\g}\right)\gtrsim \left( \sum\limits_{k=N}^{MN}  \frac{1}{(k+1)} -\frac{1}{(MN)^{2-\g}} \sum\limits_{k=N}^{MN}(k+1)^{1-\g} \right) ,$$
and due to $\frac{1}{(MN)^{2-\g}} \sum\limits_{k=N}^{MN}(k+1)^{1-\g}$ is uniformly bounded for all $M, N \in \N$, there exist $C_1=C_1(\g)>0$ and $C_2=C_2(\g)>0$ such that
$$ \frac{1}{\sum\limits_{n=0}^{MN}(n+1)^{1-\g}}\left(\sum\limits_{k=N}^{MN}  \frac{1}{(k+1)} \sum\limits_{n=k}^{MN}(n+1)^{1-\g}\right)>  C_1 \log M -C_2, \; M, N \in \N.$$
Then take a sufficiently large $M \in \N $ such that $\log M>\frac{2C+C_2}{C_1}$ so that there exists $C'=C'(\om,\g)>1$ and $M=M(\om,\g)>1$ such that $\om_N\geq C'\om_{MN}$ for all $N \in \N$. This is $\om \in \M$.

\end{Prf}

\begin{Prf}\textit{Theorem \ref{weighted Bergman spaces}.}

Assume $\om \in \DDD$ and note that it is enough proving that there exists a constant $C>0$ such that $\| C_{\om}(f)\|_{A^2_{\mu}}^2\leq C \|f\|_{A^2_{\mu}}^2$ for any function $f \in \H(\D)$ such that $\fg(n)\geq 0$, $n \in \N \cup \{0\}$. 

By following the proof of Theorem \ref{weighted H spaces} we obtain there exists $0<\b<1$ such that 
\begin{equation*}
 \| C_{\om}(f)\|_{A^2_{\mu}}^2 \lesssim \sum_{n=0}^\infty \frac{\mu_{2n+1}}{(n+1)^{2\b}} \left(\sum_{k=0}^n\frac{\widehat{f}(k)}{(n-k+1)^{1-\b}}\right)^2,
\end{equation*}
and by Lemma \ref{caract M}(iii),

$$ \| C_{\om}(f)\|_{A^2_{\mu}}^2 \lesssim \sum\limits_{n=0}^{\infty}(\mu_{[2\b]})_{2n+1} \left(\sum\limits_{k=0}^n \fg (k) \a_{n-k}\right)^2 \lesssim \left\| \frac{f(z)}{(1-z)^{\b}}\right\|_{A^2_{\mu_{[2\b]}}}^2 \lesssim \|f \|^2_{A^2_{\mu}},$$
where we recall that $\a_n$, $n \in \N$, denote the Taylor coefficients of the function $g(z)=\frac{1}{(1-z)^{\b}}.$

Reciprocally, let $\mu \in \DD$ and assume $C_{\om}: A^2_{\mu}\to A^2_{\mu}$ is bounded. Firsty we will show $\om \in \DD$. Now we consider the following family of functions $f_{N}(z)=\sum\limits_{n=0}^{N} (\mu_{2n+1})^{-\frac{1}{2}} z^n$, $N \in \N$. Then $\|f_N\|_{A^2_{\mu}}^2 \asymp (N+1)$ and on the other hand
\begin{equation*}
\begin{split}
\| C_{\om}(f_N)\|_{A^2_{\mu}}^2 & \geq \sum\limits_{n=4N}^{5N}  \mu_{2n+1} \om_n^2 \left(\sum_{k=0}^N \frac{\mu_{2k+1}^{-\frac{1}{2}}}{2(n-k+1)\om_{2(n-k)+1}}\right)^2
\\ & \gtrsim
\frac{\om_{5N}^2 }{\om_{6N}^2} \frac{1}{(N+1)^2}\sum_{n=4N}^{5N}\mu_{2n+1} \left(\sum_{k=0}^N \mu_{2k+1}^{-\frac{1}{2}} \right)^2,
\end{split}
\end{equation*}
for all $N \in \N$, so Lemma \ref{caract. pesos doblantes}(iv) yields there exists $\a=\a(\mu)>0$ such that 
$$ \| C_{\om}(f_N)\|_{A^2_{\mu}}^2 \gtrsim \frac{\om_{5N}^2 }{\om_{6N}^2} \frac{1}{(N+1)}\frac{\mu_{10N+1} }{\mu_{2N+1}}  \left(\sum_{k=0}^N \left(\frac{2k+1}{2N+1}\right)^{\frac{\a}{2}}\right)^2 \gtrsim (N+1)\frac{\om_{5N}^2 }{\om_{6N}^2}.$$

The boundedness of $C_{\om}$ yields $\om_{5N}\lesssim \om_{6N}$, for all $N \in \N$ and this implies $\om \in \DD$ by Lemma \ref{caract. pesos doblantes} (v).

Next we will prove $\om \in \M$. Consider the family of test functions 
$f_{N,M} (z)=\sum\limits_{n=0}^{MN}(\mu_{2n+1})^{-\frac{1}{2}}z^n,$  $ N, \, M \in \N$. As before we obtain $\|f_{N, M}\|_{A^2_{\mu}}^2 \asymp (MN+1)$ and on the other hand, by Lemma \ref{caract. pesos doblantes}(iv) there exists $\a>2$ such that
\begin{equation*}
\begin{split}
\| C_\omega(f_{N,M})\|^2_{A^2_{\mu}}&
\gtrsim \om_{MN}^2  \sum\limits_{n=0}^{MN} \mu_{2n+1} \left(\sum_{k=0}^n \frac{(\mu_{n-k+1})^{-\frac{1}{2}}}{(k+1)\omega_{k+1}}\right)^2 
\\ &\gtrsim \om_{MN}^2  \sum\limits_{n=0}^{MN} \frac{1}{(2n+1)^{\a}} \left(\sum_{k=0}^n \frac{(n-k+1)^{\frac{\a}{2}}}{(k+1)\omega_{k+1}}\right)^2,
\end{split}
\end{equation*}
for all $M,N \in \N$. So Jensen inequality yields
\begin{equation*}
 \left(\frac{1}{MN+1} \sum\limits_{n=0}^{MN} \frac{1}{(2n+1)^{\frac{\a}{2}}} \sum_{k=0}^n \frac{(n-k+1)^{\frac{\a}{2}}}{(k+1)\omega_{k+1}}\right)^2 \lesssim \frac{1}{\om_{MN}^2}, \; M,N \in \N.
\end{equation*}
Therefore, 
$$
\frac{\om_{MN}}{\om_N} \sum\limits_{k=N}^{MN}\frac{1}{k+1}  \sum_{n=k}^{MN} \frac{(n-k+1)^{\frac{\a}{2}}}{(n+1)^{\frac{\a}{2}}} \leq C (MN+1).
$$
To complete the proof, we will show that there exists $M \in \N $ large enough such that $$\frac{1}{C}\frac{1}{MN+1}\sum\limits_{k=N}^{MN}\frac{1}{k+1}  \sum_{n=k}^{MN} \frac{(n-k+1)^{\frac{\a}{2}}}{(n+1)^{\frac{\a}{2}}}>2.$$
Since $\a\geq 2$, $(n-k+1)^{\frac{\a}{2}}\geq (n+1)^{\frac{\a}{2}}(1-\frac{k+1}{n+1})^{\frac{\a}{2}} \geq (n+1)^{\frac{\a}{2}}- \frac{\a}{2}(k+1)(n+1)^{\frac{\a}{2}-1}$  for all $n \geq k$, so
\begin{equation*}
\begin{split}
&\frac{1}{MN+1}\sum\limits_{k=N}^{MN}\frac{1}{k+1}  \sum_{n=k}^{MN} \frac{(n-k+1)^{\frac{\a}{2}}}{(n+1)^{\frac{\a}{2}}}
 \\
&\geq \frac{1}{MN+1}\left(  \sum_{k=N}^{MN} \frac{MN-k+1}{k+1}-\frac{\a}{2}\sum_{k=N}^{MN}\sum\limits_{n=k}^{MN} \frac{1}{n+1}\right)\asymp  \sum_{k=N}^{MN} \frac{1}{k+1} -1
\end{split}
\end{equation*}
 and  there exist $C_1, C_2>0$ such that
\begin{equation*}
\begin{split}
\frac{1}{MN+1}\sum\limits_{k=N}^{MN}\frac{1}{k+1}  \sum_{n=k}^{MN} \frac{(n-k+1)^{\frac{\a}{2}}}{(n+1)^{\frac{\a}{2}}}\geq C_1 \log M - C_2
\end{split}
\end{equation*}
Therefore, for a fixed $M \in \N $ such that $\log M>\frac{2C+C_2}{C_1}$, there exists $C'>1$ such that $\om_N\geq C'\om_{MN}$ for all $N \in \N$. Then $\om \in \M$ and the proof is finished.

\end{Prf}

\section{Compactness}

\begin{lemma}\label{LemmaComp}

Let $\omega$ be a radial weight and  $\{f_k\}_{k=0}^\infty\subset \mathcal{H}(\mathbb{D})$ such that $f_k\to 0$ uniformly on compact subsets of $\D$. Then, $C_{\om}(f_k)\to 0$ uniformly on compact subsets of $\D$.
\end{lemma}
\begin{proof}
Let be $M\subset \D$ a compact subset and $K_t^\om(z)=\frac1z\int_0^z B_t^\om(u)\,du$. If $z\in M$,
$$ |C_\om(f_k)(z)|\le \int_0^1 |f_k(tz)| |K_t^\om(z)|\om(t)\,dt.$$
By following the proof of \cite[Lemma 20]{MerPeldelaR}, we obtain that there exists a $\rho_0\in(0,1)$ such that $M\subset \overline{D(0,\rho_0)}$ and
$$ \sup_{\substack{z\in M\\ t\in[0,1)}} |K_t^\omega(z)|\le C(\om,\rho_0)<\infty.$$

Let $\varepsilon>0$. By hypothesis, there exists a $k_0\in\N$ such that for every $k\ge k_0$ and $tz\in \overline{D(0,\rho_0)}$, $|f_k(tz)|<\varepsilon$. Putting all together, we have that
$|C_\om(f_k)(z)|\le \varepsilon\cdot C(\om,\rho_0)\cdot \omega_0$, so $C_\om(f_k)\to 0$ uniformly on $M$.
\end{proof}

Bearing in mind the previous lemma and by following a classic argument (see for example \cite[Theorem 21]{MerPeldelaR}) we claim the following characterization of the compactness holds.

\begin{theorem}\label{thmAuxComp}
Let $\omega$ and $\mu$ be  radial weights, $\gamma>0$, $\mu \in \mathcal{D}$, and $X\in \lbrace \mathcal{H}_{\gamma} , A^{2}_{\mu}\rbrace $. Then, the following assertions are equivalent:
\begin{itemize}
\item[(i)]$\cw:X \to X$ is compact;
\item[(ii)]For every sequence $\{f_k\}_{k=0}^\infty\subset X$ such that $\sup\limits_{k\in\N}\|f_k\|_{X}<\infty$ and $f_k\to 0$ uniformly on compact subsets of $\D$, $\lim\limits_{k\to\infty} \|\cw(f_k)\|_{X}=0$.
\end{itemize}
\end{theorem}

Once we have the previous result, we are able to show that there does not exist radial weight $\om$ such that $\cw:\hg\to\hg$, $\g >0$ is compact neither $\cw:A^{2}_{\mu}\to A^{2}_{\mu}$, $\mu \in \DDD$, is compact.

\begin{theorem}
\label{compact}
Let $\omega$ and $\mu$ be  radial weights, $\gamma>0$, $\mu \in \mathcal{D}$, and $X\in \lbrace \mathcal{H}_{\gamma} , A^{2}_{\mu}\rbrace $. Then, $\cw: X \to X$ is not compact.
\end{theorem}

\begin{proof}
\textbf{Case $X=\hg$.}
For each $a\in(0,1)$ we set
$$ f_a(z)=\sum_{n=0}^{\infty}(1-a^2)^{\frac{1}{2}}\frac{a^n}{(n+1)^{\frac{1-\gamma}{2}}} z^n.$$ 
Consequently, it is obvious that
$$
\|f_a\|_{\hg}^2\asymp \sum\limits_{n=0}^\infty |\widehat{f}_a(n)|^2(n+1)^{1-\gamma}
= \sum\limits_{n=0}^\infty (1-a^2)a^{2n}=1, \quad a\in(0,1).$$
Furthermore, it is clear that $f_a\to 0$ as $a\to1^-$ uniformly on compact subsets of $\D$. In addition,  we have
\begin{align*}
\|\cw(f_a)\|_{\hg}^2&\asymp \sum\limits_{n=0}^\infty \omega_n^2\left( \sum\limits_{k=0}^n \frac{\widehat{f}_a(k)}{2(n-k+1)\omega_{2(n-k)+1}}\right)^2(n+1)^{1-\gamma}
\\ 
&\asymp (1-a^2) \sum\limits_{n=0}^\infty \omega_n^2\left( \sum\limits_{k=0}^n \frac{a^k \left( k+1\right)^{\frac{\gamma-1}{2}}}{(n-k+1)\omega_{2(n-k)+1}}\right)^2(n+1)^{1-\gamma}
\\ 
&\ge (1-a^2)\sum\limits_{n=0}^\infty \frac{\omega_n^2 a^{2n}}{(n+1)^{1+\gamma}}\left( \sum\limits_{k=0}^n \frac{\left( k+1\right)^{\frac{\gamma-1}{2}}}{\omega_{2(n-k)+1}}\right)^2
\\ 
&\gtrsim (1-a^2)\sum\limits_{n=0}^\infty \frac{\omega_{2n}^2 a^{4n}}{(n+1)^{1+\gamma}}\left( \sum\limits_{k=0}^n \frac{\left( k+1\right)^{\frac{\gamma-1}{2}}}{\omega_{2(2n-k)+1}}\right)^2
\\ 
&\ge (1-a^2) \sum\limits_{n=0}^\infty \frac{a^{4n}}{(n+1)^{1+\gamma}}\left( \sum\limits_{k=0}^n \left( k+1\right)^{\frac{\gamma-1}{2}}\right)^2
\\
 &\asymp 1,
\end{align*}
so using Theorem \ref{thmAuxComp} we deduce that $\cw:\hg\to\hg$ is not a compact operator.

\textbf{Case $X=A^2_{\mu}$.} For each $a\in(0,1)$ we consider
$$ f_a(z)=\sum_{n=0}^{\infty}(1-a^2)^{\frac{1}{2}}\mu_{2n+1}^{-\frac{1}{2}} a^n z^n.$$ As a result $
\|f_a\|_{A^{2}_{\mu}}^2\asymp 1,$ $a\in(0,1)$ and it is clear that $f_a\to 0$ as $a\to1^-$ uniformly on compact subsets of $\D$. By following the argument of the previous case it is not difficult to show 
$$ \|\cw(f_a)\|_{A^{2}_{\mu}}^2 \gtrsim  (1-a^2) \sum\limits_{n=0}^\infty \frac{a^{4n}}{(n+1)^{2}}\left( \sum\limits_{k=0}^n \mu_{2k+1}^{-\frac{1}{2}}\right)^2\mu_{2n+1}, \quad a \in (0,1),$$
so Lemma \ref{caract. pesos doblantes}(iv) yields that there exists $\a=\a(\mu)>0$ such that 
$$
\|\cw(f_a)\|_{A^{2}_{\mu}}^2 \gtrsim  (1-a^2) \sum\limits_{n=0}^\infty \frac{a^{4n}}{(n+1)^{2}}\left( \sum\limits_{k=0}^n \left(\frac{2k+1}{2n+1}\right)^{\frac{\alpha}{2}}\right)^2
\asymp 1.$$

Therefore, using Theorem \ref{thmAuxComp} again we deduce that $\cw:A^{2}_{\mu} \to A^{2}_{\mu}$ is not a compact operator.
\end{proof}

\end{document}